\documentclass[11pt]{amsart}
\usepackage{amsmath,amsthm,amssymb}
\usepackage[all]{xy}

\begin{document}

\newtheorem{theorem}{Theorem}[section]
\newtheorem{lemma}[theorem]{Lemma}
\newtheorem{proposition}[theorem]{Proposition}
\newtheorem{corollary}[theorem]{Corollary}
\newtheorem{bigtheorem}{Theorem}

\theoremstyle{definition}
\newtheorem{definition}[theorem]{Definition}
\newtheorem{example}[theorem]{Example}
\newtheorem{formula}[theorem]{Formula}
\newtheorem{nothing}[theorem]{}

\theoremstyle{remark}
\newtheorem{remark}[theorem]{Remark}

\renewcommand{\arraystretch}{1.2}

\newcommand{\dual}{^{\vee}}
\newcommand{\contr}{{\mspace{1mu}\lrcorner\mspace{1.5mu}}}
\newcommand{\de}{\partial}
\newcommand{\debar}{{\overline{\partial}}}
\newcommand{\pibar}{\overline{\pi}}

\newcommand{\desude}[2]{{\dfrac{\de #1}{\de #2}}}

\newcommand{\mapor}[1]{{\stackrel{#1}{\longrightarrow}}}
\newcommand{\ormap}[1]{{\stackrel{#1}{\longleftarrow}}}

\newcommand{\mapver}[1]{\Big\downarrow\vcenter{\rlap{$\scriptstyle#1$}}}

\newcommand{\binfty}{\boldsymbol{\infty}}
\newcommand{\bi}{\boldsymbol{i}}
\newcommand{\bl}{\boldsymbol{l}}

\renewcommand{\bar}{\overline}
\renewcommand{\Hat}[1]{\widehat{#1}}

\newcommand{\sA}{\mathcal{A}}
\newcommand{\Oh}{\mathcal{O}}
\newcommand{\sH}{\mathcal{H}}
\newcommand{\sL}{\mathcal{L}}
\newcommand{\sM}{\mathcal{M}}
\newcommand{\sB}{\mathcal{B}}
\newcommand{\sY}{\mathcal{Y}}

\newcommand{\Q}{\mathbb{Q}}
\newcommand{\C}{\mathbb{C}}
\newcommand{\K}{\mathbb{K}}
\newcommand{\Proj}{\mathbb{P}}

\newcommand{\DER}{{{\mathcal D}er}}

\newcommand{\ad}{\operatorname{ad}}
\newcommand{\MC}{\operatorname{MC}}
\newcommand{\Def}{\operatorname{Def}}
\newcommand{\Hom}{\operatorname{Hom}}
\newcommand{\End}{\operatorname{End}}
\newcommand{\Image}{\operatorname{Im}}
\newcommand{\image}{\operatorname{Im}}
\newcommand{\Der}{\operatorname{Der}}
\newcommand{\Mor}{\operatorname{Mor}}
\newcommand{\Cone}{\operatorname{Cone}}
\newcommand{\Aut}{\operatorname{Aut}}
\newcommand{\coker}{\operatorname{Coker}}

\newcommand{\Grass}{\operatorname{Grass}}
\newcommand{\Spec}{\operatorname{Spec}}

\newcommand{\Id}{\operatorname{Id}}

%\newcommand{\dee}{q_+}

%%%%%%%%%%%%%%%%%%%%%%%%%%%%%%

\title
{A period map for generalized deformations}
\date{July 25, 2008}

\author{Domenico Fiorenza}
\address{\newline Dipartimento di Matematica ``Guido
Castelnuovo'',\hfill\newline
Universit\`a di Roma ``La Sapienza'',\hfill\newline
P.le Aldo Moro 5,
I-00185 Roma Italy.\hfill\newline}
\email{fiorenza@mat.uniroma1.it}
\urladdr{www.mat.uniroma1.it/\~{}fiorenza/}
\author{Marco Manetti}
\email{manetti@mat.uniroma1.it}
\urladdr{www.mat.uniroma1.it/people/manetti/}

\begin{abstract}
For every compact K\"{a}hler manifold we give a canonical
extension of Griffith's period map to generalized deformations,
intended as solutions of Maurer-Cartan equation in the algebra of
polyvector fields. Our construction involves the notion of Cartan
homotopy and a canonical $L_{\infty}$ structure on mapping cones of
morphisms of differential graded Lie algebras.
\end{abstract}


\maketitle

\section*{Introduction}

Let $X$ be a compact K\"{a}hler manifold and denote by $H^*(X,\C)$
the graded vector space of its De Rham cohomology. The goal of
this paper is to define a natural transformation
\[\Phi\colon \Def_X\to \Aut_{H^*(X,\C)}\]
from infinitesimal deformations of $X$ to automorphisms of
$H^*(X,\C)$. More precisely, for every local Artinian $\C$-algebra
$A$ and every deformation  of $X$ over $\Spec(A)$ we define in a
functorial way a canonical morphism of schemes
\[ \Spec(A)\to GL(H^*(X,\C))=\prod_n GL(H^n(X,\C)).\]

Our construction will be carried out by using the interplay of
Cartan homotopies and $L_{\infty}$-morphisms  and it is compatible
with classical  construtions of the theory of infinitesimal variations of Hodge structures. In particular:
\begin{enumerate}

\item Via the natural isomorphism $H_X\simeq
\oplus_{p,q}H^q(\Omega^p_X)$ induced by Dolbeault's theorem
and the
$\de\debar$-lemma, the differential of $\Phi$
\[
d\Phi\colon H^1(T_X)\to \Hom^0(H^*(X,\C),H^*(X,\C))
\]
is identified with the
contraction operator: $d\Phi(\xi)=\bi_{\xi}$,
where $\bi_{\xi}(\omega)=\xi\contr\omega$.

\item The contraction
\[ \bi\colon H^2(T_X)\to \Hom^1(H^*(X,\C),H^*(X,\C))\]
is a morphism of obstruction theories. In particular, since
$\Aut_{H^*(X,\C)}$ is smooth, every obstruction to deformation of
$X$ is contained in the kernel of $\bi$.

\item For every $m$ let $H^*(F^m)\subseteq H^*(X,\C)$ be  the subspace
of cohomology classes of closed $(p,q)$-forms, with $p\ge m$. Then
the composition of $\Phi$ with the natural projection
\[ GL(H^*(X,\C))\to \Grass(H^*(X,\C))=\prod_n \Grass(H^n(X,\C)),\qquad f\mapsto
f(H^*(F^m)),\] is the classical period map.

\end{enumerate}

We will define and study the morphism $\Phi$ using the framework
of $L_{\infty}$-algebras. It is however useful to give also a more
geometric definition in the following way.\par

Denote by $A_X=\oplus_{i}A^{i}_X$ the space of complex valued
differential forms on $X$, by $d=\de+\debar\colon A_X^i\to
A_X^{i+1}$ the De Rham differential and by $\de A_X\subseteq A_X$
the subspace of $\de$-exact forms. A small variation of the almost
complex structure is determined by a form $\xi\in A_X^{0,1}(T_X)$:
according to Newlander-Niremberg theorem, the integrability
condition of $\xi$  is equivalent to $(d+\bl_{\xi})^2=0$, where
$\bl_{\xi}\colon A^i_X\to A^{i+1}_X$ is the holomorphic Lie
derivative, defined by the formula
\[ \bl_{\xi}(\omega)=\de(\xi\contr\omega)+\xi\contr \de\omega.\]
Notice that $\bl_{\xi}(\ker\de)\subseteq \de A_X$.\par

Assume therefore $(d+\bl_{\xi})^2=0$; according to
$\de\debar$-lemma, the complex $(\de A_X,d)$ is
acyclic and then, if $\xi$ is sufficiently small, the complex
$(\de A_X,d+\bl_{\xi})$ is still acyclic.\par

In order to define the automorphism $\Phi_{\xi}\colon
H^*(X,\mathbb{C})\to H^*(X,\mathbb{C})$,   let $[\omega]\in
H^*(X,\mathbb{C})$ and choose a $d$-closed form
$\omega_0^{}\in A_X$ representing $[\omega]$ and such that
$\de\omega_0^{}=0$. Since
\[ (d+\bl_{\xi})\omega_0^{}=\de(\xi\contr\omega_0^{})\in \de
A_X
\quad\text{and}\quad (d+\bl_{\xi})^2\omega_0^{}=0,\] there
exists $\beta\in A_X$ such that
$(d+\bl_{\xi})\omega_0^{}=(d+\bl_{\xi})\de\beta$. If
$\bi_{\xi}$ is the contraction, then it is not difficult to
prove that
$d(e^{\bi_{\xi}}(\omega_0^{}-\de\beta))=0$ and the
cohomology class of
$e^{\bi_{\xi}}(\omega_0^{}-\de\beta)$ does not depend on
the choice of $\beta$ and $\omega_0^{}$, allowing to define
$\Phi_{\xi}([\omega])$  as the cohomology class of
$e^{\bi_{\xi}}(\omega_0^{}-\de\beta)$.

Equivalently, for every $d$-closed form $\omega\in A_X$ and every
small variation of the complex structure $\xi$ we have
\[
\Phi_{\xi}([\omega])=[e^{\bi_{\xi}}(\omega-d\gamma-\de\beta)],\quad\text{
where }\quad \de\omega=\de d\gamma\, ,\quad
\bl_{\xi}(\omega-d\gamma)=(d+\bl_{\xi})\de\beta.\]

Moreover, as direct consequence of the $L_{\infty}$ approach,
we will see that $\Phi_{\xi}$ is invariant under the gauge
action, where two integrable small variation of the almost
complex structure
$\xi_1,\xi_2$ are gauge equivalent if and only if they give
isomorphic deformations of $X$.

Our construction generalizes in a completely straightforward way
to \emph{generalized deformations} of $X$,
defined as the solutions,
up to gauge equivalence, of the Maurer-Cartan equation in the
differential graded Lie algebra
\[\operatorname{Poly}_X=\bigoplus_i
\operatorname{Poly}^{i}_X,\qquad \operatorname{Poly}^{i}_X=
\bigoplus_{b-a=i-1} A_X^{0,b}(\wedge^{a}T_X),\] endowed with the
opposite of Dolbeault differential and the Schouthen-Nijenhuys
bracket.

Putting together all these fact, at the end we get, for every $m$,
a commutative diagram of morphism of functors of Artin rings
\[\xymatrix{\widetilde{\Def}_X\ar[r]^{\Phi}&
\Aut_{H^*(X,\C)}\mspace{-30mu}\ar[d]^{\pi}\\
\Def_X\ar[u]^{i}\ar[r]^{p\quad}&\Grass_{H^*(F^m),H^*(X,{\mathbb
C})}\mspace{-90mu}}\] where $\widetilde{\Def}_X$ is the functor of
generalized deformations of $X$, $i$ is the natural inclusion,
$\Grass_{H^*(F^m),H^*(X,{\mathbb C})}$ is the Grassmann functor
with base point $H^*(F^m)$, $\pi$ is the smooth morphism defined
as $\pi(f)=f(H^*(F^m))$ and $p$ is the classical $m^{\text{th}}$
period map.\par

In view of this result is natural to candidate the composition
$\pi\Phi$ as period map for generalized deformations.

\begin{example} Our definition  is compatible with yet existing
notion of period  map for generalized deformations of Calabi-Yau
manifolds used in some mirror symmery constructions.\par

In fact, if  $X$ is a Calabi-Yau manifold with volume element
$\Omega$, then by Tian-Todorov Lemma , every generalized
deformation over $\Spec(A)$ is represented by an element $\xi\in
\operatorname{Poly}^{1}_X\otimes\mathfrak{m}_A$ such that
\[ D\xi+\frac{1}{2}[\xi,\xi]=0,\qquad \de(\xi\contr\Omega)=0.\]
Under these assumptions our recipe  gives
\[ \Phi_{\xi}([\Omega])=[e^{\bi_{\xi}}(\Omega)]\]
and therefore we recover the construction of \cite{BK}.
\end{example}

\textbf{Keywords and general notation.}

We assume that the reader is familiar with the notion and main
properties
of differential graded Lie algebras and
$L_{\infty}$-algebras (we refer to
\cite{fuka,K,LadaMarkl,LadaStas,defomanifolds}
as introduction of such structures); however
the basic definitions are recalled in this paper in order to fix
notation and terminology. 
For the whole paper, $\mathbb{K}$ is a field of characteristic 0;
every vector space is intended  over $\mathbb{K}$. $\mathbf{Art}$
is the category of local Artinian $\K$-algebras with residue field
$\K$. For $A\in\mathbf{Art}$ we denote by $\mathfrak{m}_A$ the
maximal ideal of $A$. By abuse of notation, if $F\colon
\mathbf{Art}\to \mathbf{Set}$ is a functor, we write $\xi\in F$ to
mean $\xi\in F(A)$ for some fixed $A\in \mathbf{Art}$.

\section{Deformation functors associated with DGLA morphisms}
\label{sec:dgla-morphism}

We recall from \cite{semireg} that, to any morphism of
differential graded Lie algebras over a field $\K$ of
characteristic $0$, $\chi\colon L\to M$, are naturally associated
two functors of Artin rings $\MC_{\chi},\Def_{\chi}\colon
\mathbf{Art}\to \mathbf{Set}$,
%where
%$\mathbf{Art}$ is the category of local Artinian $\K$-algebras
%with residue field $\K$,
in the following way:
\[ \MC_{\chi}(A)=
\left\{(x,e^a)\in (L^1\otimes\mathfrak{m}_A)\times \exp(M^0\otimes
\mathfrak{m}_A)\mid
dx+\frac{1}{2}[x,x]=0,\;e^a\ast\chi(x)=0\right\},\]
\[ \Def_\chi(A)=\frac{\MC_{\chi}(A)}{\text{gauge equivalence}},\]
where two solutions of the Maurer-Cartan equation are gauge
equivalent if they belong to the same orbit of the gauge action
\[(\exp(L^0\otimes\mathfrak{m}_A)\times
\exp(dM^{-1}\otimes\mathfrak{m}_A))
\times \MC_{\chi}(A)\mapor{\ast}\MC_{\chi}(A)\]%
given by the formula
\[ (e^l, e^{dm})\ast (x,e^a)=(e^l\ast x,
e^{dm}e^ae^{-\chi(l)})=(e^l\ast x,
e^{dm\bullet a\bullet (-\chi(l))}).\]%
The $\bullet$ in the rightmost term in the above formula is the
Baker-Campbell-Hausdorff multiplication; namely
$e^xe^y=e^{x\bullet y}$. Note that if $L=0$ and the differential
on $M$ is trivial, then
\[
\Def_\chi(A)=\exp(M^0\otimes
\mathfrak{m}_A).
\]
\par
It has been shown in \cite{cone} that the suspended cone of
$\chi$, i.e., the differential complex $(C_\chi,\mu_1)$ given by
the graded vector space
\[ C_\chi=\mathop{\oplus}_iC_\chi^i,\qquad C_\chi^i=
L^i\oplus M^{i-1},\]%
endowed with the differential
\[ \mu_1(l,m)=(dl,\chi(l)-dm),\qquad l\in L,m\in M,\]%
carries a natural compatible $L_\infty$-algebra structure, which
we shall denote $\widetilde{C}(\chi)$, such that the associated
deformation functor $\Def_{\widetilde{C}(\chi)}$ is naturally
isomorphic to $\Def_{\chi}$. More precisely, the map $(l,m)\mapsto
(l,e^m)$ induces a natural isomorphism
$\MC_{\widetilde{C}(\chi)}\xrightarrow{\sim}\MC_{\chi}$, and
homotopy eqivalence on $\MC_{\widetilde{C}(\chi)}$ is identified
with gauge equivalence on $\MC_{\chi}$.
\par
The higher brackets
\[\mu_n\colon\bigwedge^nC_{\chi}\to
C_{\chi}[2-n],\qquad n\ge 2,\]
 defining the $L_\infty$-algebra structure $\widetilde{C}(\chi)$ have been
explicitly described in \cite{cone}. Namely, one has
\[ \mu_2((l_1,m_1)\wedge (l_2,m_2))=
\left([l_1,l_2],\frac{1}{2}[m_1,\chi(l_2)]+\frac{(-1)^{\deg(l_1)}}{2}[\chi(l_1),m_2]\right)\]
and for $n\ge 3$
\[
\mu_n((l_1,m_1)\wedge\cdots\wedge(l_n,m_n))=\pm\frac{B_{n-1}}{(n-1)!}\sum_{\sigma\in
S_n}\varepsilon(\sigma)
[m_{\sigma(1)},[\cdots,[m_{\sigma(n-1)},\chi(l_{\sigma(n)})]\cdots]]
\]
Here the $B_n$'s are the Bernoulli numbers, $\varepsilon$ is
the Koszul sign and we refer to
\cite{cone} for the exact determination of the overall
$\pm$ sign in the above formulas (it will not be needed in
the present paper). Note that the projection on the first
factor
$\pi_1\colon
\widetilde{C}(\chi)\to L$ is a linear $L_\infty$-morphism.
\par
By the functoriality of $\chi\mapsto \widetilde{C}(\chi)$
if
\[
\xymatrix{
      L_1 \ar[r]^{f_L}
\ar[d]_{\chi_1} &
L_2 \ar[d]^{\chi_2}\\
    M_1  \ar[r]^{f_M} &
M_2\\
    }
\]
is a a commutative diagram of morphisms of differential graded
Lie algebras, then $(x,e^a)\mapsto(f_L(x),e^{f_M(a)})$ is a natural
 transformation of Maurer-Cartan functors
inducing a natural transformation
\[
\Def_{\chi_1}\to \Def_{\chi_2}.
\]
Moreover, if $f_L$ and $f_M$ are quasi-isomorphisms, then
$\Def_{\chi_1}\xrightarrow{\sim} \Def_{\chi_2}$ is an isomorphism.

\section{An example from K\"ahler
geometry}\label{sec.example-kaehler}

Let $X$ be a compact K\"ahler manifold. Consider the DGLA
$\Hom^*(A_X,A_X)$ of graded endomorphisms of the De Rham complex
and their subDGLAs
\[ L=\{f\in \Hom^*(A_X,A_X)\mid f(\ker\partial)\subseteq
\partial A_X\},\]
\[ M=\{f\in \Hom^*(A_X,A_X)\mid f(\ker\partial)\subseteq
\ker\partial\text{ and } f(\partial A_X)\subseteq \partial
A_X\}.\]

Then we have a commutative diagram of morphisms of DGLAs,
where the vertical arrows are the inclusions
\[
\xymatrix{0\ar[d]_{\rho}&
      L \ar@{=}[r]
\ar[d]_{\eta} \ar[l]&
L \ar[d]^{\chi}\\
\Hom^*\left(\dfrac{\ker\de}{\de A_X},\dfrac{\ker\de}{\de
A_X}\right) & M \ar[r]\ar[l]&
\Hom^*(A_X,A_X)\\
    }
\]
By the $\de\debar$-lemma, we have quasi-isomorphisms
\[ (A_X,d)\longleftarrow (\ker\de, d)\longrightarrow
\left(\frac{\ker\de}{\de A_X},0\right)\cong (H^*(X,\mathbb{C}),0).\]
Hence the
horizontal arrows in the above commutative diagram are
quasi-isomorphisms and we get the following isomorphisms of deformation functors
\[
\Def_\chi\simeq\Def_\eta\simeq
\Def_\rho=\Aut_{H^*(X,\mathbb{C})}.
\]
The isomorphism
$\psi\colon\Def_\chi\simeq\Aut_{H^*(X,\mathbb{C})}$ is
explicitly described as follows: given a Maurer-Cartan
element $(\alpha,e^a)\in \MC_\chi$ and a cohomology class
$[\omega]\in H^*(X,\mathbb{C})$,
\[
\psi_a([\omega])=[e^{\tilde{a}}\omega_0-\de \beta_0],
\]
where $(\tilde{\alpha},e^{\tilde{a}})\in \MC_\eta\subseteq
\MC_\chi$ is gauge-equivalent to $(\alpha,e^a)$, the differential
form $\omega_0$ is a $\de$-closed representative for the
cohomology class $[\omega]$, and $\beta_0\in A_X$ is such that
$d(e^{\tilde{a}}\omega_0-\de \beta_0)=0$. The cohomology class
$[e^{\tilde{a}}\omega_0-\de \beta_0]$ is independent of the
choices of $(\tilde{\alpha},e^{\tilde{a}})$, $\omega_0$ and
$\beta_0$. Since $e^{\tilde{a}}$ is an automorphism of $\de A_X$,
we can write
\[
\psi_a([\omega])=[e^{\tilde{a}}(\omega_0-\de \beta_0)]
\]
for any $(\tilde{\alpha},e^{\tilde{a}})$ and $\omega_0$ as
above and any $\beta\in A_X$ such that
$de^{\tilde{a}}(\omega_0-\de \beta)=0$.

\begin{remark}\label{rem.differential-psi}
Since $\chi\colon L\hookrightarrow\Hom^{*}(A_X,A_X)$ is
injective, the projection on the second factor $C_\chi\to
\Hom^{*-1}(A_X,A_X)$ induces an identification
$H^*(C_\chi)\simeq H^{*-1}({\rm coker\, \chi})$ and so
in particular
\[
H^1(C_\chi)\xrightarrow{\sim}
H^{0}(\Hom^*(\ker\de,A_X/\de
A_X))=\Hom^0(H^*(\ker\de),H^{*}(A_X/\de
A_X)).
\]
Hence, the differential of $\psi$
is naturally identified with the linear isomorphism
\[
\Hom^0(H^*(\ker\de),H^{*}(A_X/\de
A_X))\simeq \Hom^0(H^*(X,\mathbb{C});H^*(X,\mathbb{C}))
\]
induced by the $\de\debar$-lemma and the De Rham isomorphism.
\end{remark}

For later use, we  give a  more explicit description of the map
$\psi$ by writing a map $\widetilde\psi\colon\MC_\chi\to
\Aut_{H^*(X,\mathbb{C})}$ inducing it. To  define the map
$\widetilde\psi$ we need a few preliminary remarks.

\begin{lemma}\label{lem.newdifferential}
If $(\alpha,e^a)\in \MC_\chi$, then in the associative
algebra
$\Hom^*(A_X,A_X)$ we have the equality
\[ e^{-a} d e^{a}=d+\alpha.\]
\end{lemma}
\begin{proof}
By the definition of the gauge action in the DGLA
$\Hom^*(A_X,A_X)$, one has for every
$x\in \Hom^0(A_X,A_X)$, $y\in \Hom^1(A_X,A_X)$ the formula
\[ e^x\ast y=e^x(d+y)e^{-x}-d.\]
In particular
$e^{-a} d e^{a}=
d+e^{-a}\ast 0$.
\end{proof}

\begin{corollary} If $(\alpha,e^a)\in \MC_\chi$, then the
graded subspaces
\[ e^{a}(\ker\de),\quad  e^{a}(\de A_X)\]
are subcomplexes of $(A_X,d)$.
Moreover, the map
\[ e^{a}\colon
\left(\frac{\ker\de}{\de A_X},0\right)\to
\left(\frac{e^{a}(\ker\de)}{e^{a}(\de
A_X)},0\right)\] is an isomorphism of complexes
and  the natural maps
\[ (A_X,d)\longleftarrow (e^{a}(\ker\de), d)\longrightarrow
\left(\frac{e^{a}(\ker\de)}{e^{a}(\de
A_X)},0\right)\] are quasiisomorphisms.
\end{corollary}

\begin{proof}
Both $(\de A_X,d)$ and $(\ker\de,d)$ are subcomplexes of
$(A_X,d)$. Since $d e^{a}(v)=e^{a}(dv+\alpha(v))$, with
$\alpha(\ker\de)\subseteq\de A_X$, we have $de^{a}(\de
A_X)\subseteq e^a(\de A_X)$ and $de^{a}(\ker\de)\subseteq e^a(
\ker\de)$. The induced differential on the quotient space
$\frac{e^{a}(\ker\de)}{e^{a}(\de A_X)}$ is trivial since, by the
$\de\debar$-lemma, $d(\ker\partial)\subseteq\de A_X$. 
Again by $\de\debar$-lemma the complex $(\de A_X,d)$ is acyclic and therefore
the morphisms of complexes
\[ (A_X,d)\longleftarrow (\ker\de, d)\longrightarrow
\left(\frac{\ker\de}{\de A_X},0\right)\] are quasiisomorphisms.
Since every 
infinitesimal perturbations of an acyclic complex is still acyclic, the complex
$(e^a(\de A_X),d)$ is acyclic and therefore 
the morphisms of complexes
\[ (A_X,d)\longleftarrow (e^a(\ker\de), d)\longrightarrow
\left(\frac{e^a(\ker\de)}{e^a(\de A_X)},0\right)\] are quasiisomorphisms.
\end{proof}

\begin{definition} The isomorphism
\[ \widetilde{\psi}_a\colon H^*(X,\mathbb{C})\to
H^*(X,\mathbb{C})\] associated to a Maurer-Cartan element
$(\alpha,e^a)$ via the natural map $\MC_\chi\to
\Def_\rho\simeq
\Aut^0(H^*(X;{\mathbb C})$ is obtained by the De Rham
isomorphism $H^*(X,\mathbb{C})=H^*(A_X,d)$ and the chain of
quasi-isomorphisms
\[\begin{array}{ccccc}
A_X&\longleftarrow&\ker\de&\longrightarrow&\dfrac{\ker\de}{\de A_X}\\
&&&&\mapver{e^{a}}\\
A_X&\longleftarrow&e^{a}(\ker\de)&\longrightarrow&
\dfrac{e^{a}(\ker\de)}{e^{a}(\de
A_X)}\end{array}\]
\end{definition}

More explicitly,
\[
\widetilde{\psi}_a([\omega])=[e^a(\omega_0-\de
\beta)]
\]
for any $\de$-closed
representative $\omega_0$ of the cohomology class
$[\omega]$, and any $\beta\in A_X$ such that
$d e^{a}(\omega_0-\de \beta)=0$.

\begin{proposition} The natural transformation $\widetilde{\psi}\colon
\MC_\chi\to \Aut_{H^*(X,\mathbb{C})}$ is gauge invariant
and therefore factors to $\psi\colon \Def_\chi\to
\Aut_{H^*(X,\mathbb{C})}$.
\end{proposition}

\begin{proof}To show that
$\widetilde{\psi}_{a\bullet(-l)}=\widetilde{\psi}_a$, we note
that, since
$l(\ker\de)\subseteq \de A_X$, we have
$e^{-l}(\ker \de)=\ker \de$, $e^{-l}(\de A_X)=\de A_X$, and
\[ e^{-l}\colon \left(\frac{\ker\de}{\de A_X},0\right)\to
\left(\frac{\ker\de}{\de A_X},0\right)
\]
is the identity. To prove that $\widetilde{\psi}_{dm\bullet
a}=\widetilde{\psi}_a$, notice that we have a commutative
diagram of morphism of complexes
\[\begin{array}{ccccc}
A_X&\longleftarrow&\ker\de&\longrightarrow&\dfrac{\ker\de}{\de A_X}\\
\mapver{e^{dm}}&&\mapver{e^{dm}}&&\mapver{e^{dm}}\\
A_X&\longleftarrow&e^{dm}(\ker\de)&\longrightarrow&
\dfrac{e^{dm}(\ker\de)}{e^{dm}(\de A_X)}\end{array}\]
which, since
$dm$ is homotopy equivalent to zero, induces the commutative
diagram of isomorphisms
\[\begin{array}{ccc}
H^*(X;{\mathbb C})&\longleftrightarrow&\dfrac{\ker\de}{\de A_X}\\
\mapver{\rm Id}&&\mapver{e^{dm}}\\
H^*(X;{\mathbb C})&\longleftrightarrow&
\dfrac{e^{dm}(\ker\de)}{e^{dm}(\de
A_X)}\end{array}\]
Finally, gauge invariance of $\widetilde{\psi}$, together
with the explicit formulae for $\widetilde{\psi}([\omega])$
and $\psi([\omega])$ written above immediately imply that
$\widetilde{\psi}$ induces $\psi$.
\end{proof}

\section{Morphisms of deformation functors associated to Cartan homotopies}
\label{sec.cartan}

In this section we formalize, under the notion of \emph{Cartan
homotopy}, a set of standard identities that often arise in
algebra and geometry \cite[Appendix B]{clemens}, and show how to
any Cartan homotopy can be canonically associated a natural
transfrormation of deformation functors.

Let $L$ and $M$ be two differential graded Lie algebras. For a
given linear map $\bi\in \Hom^{-1}(L,M)$, let $\bl\colon L\to M$
be the map defined as
\[a\mapsto
\bl_a=d\bi_a+\bi_{da}.
\]
\begin{definition}\label{def:cartan}
The map $\bi$ is called a \emph{Cartan homotopy} for $\bl$ if, for
every $a,b\in L$, we have:
\[\bi_{[a,b]}=[\bi_a,\bl_b],\qquad [\bi_a,\bi_{b}]=0.\]
When the second equation is replaced by the weaker condition
\[
\sum_{\sigma\in
S_3}\pm[\bi_{x_{\sigma(1)}},
[\bi_{x_{\sigma(2)}},\bl_{x_{\sigma(3)}}]=0,
\]
where $\pm$ is the Koszul sign, we shall say that $\bi$ is a
\emph{weak Cartan homotopy}.
\end{definition}

It is straightforward  to show that the condition
$\bi_{[a,b]}=[\bi_a,\bl_b]$ implies that $\bl$ is a morphism of
differential graded Lie algebras.

\begin{example}\label{exa.realcartan}
The name Cartan homotopy has a clear origin in differential
geometry. Namely, let $M$ be a differential manifold, ${\mathcal
X}(M)$ be the Lie algebra of vector fields on $M$, and ${\mathcal
E}nd^*(\Omega^*(M))$ be the Lie algebra of endomorphisms of the de
Rham algebra of $M$. The Lie algebra ${\mathcal X}(M)$ can be seen
as a DGLA concentrated in degree zero, and the graded Lie algebra
${\mathcal E}nd^*(\Omega^*(M))$ has a degree one differential
given by $[d_{\rm dR},-]$, where $d_{\rm dR}$ is the de Rham
differential. Then the contraction
\[
\bi\colon {\mathcal X}(M)\to {\mathcal E}nd^*(\Omega^*(M))[-1]
\]
is a Cartan homotopy and its differential is the Lie derivative
\[
[d,\bi]={\mathcal L}\colon {\mathcal X}(M)\to {\mathcal
E}nd^*(\Omega^*(M)).
\]
In fact, by classical Cartan's homotopy formulas
\cite[Section 2.4]{AbrahamMarsden}, for any two vector fields $X$ and
$Y$ on $M$, we
have
\begin{enumerate}
\item ${\mathcal L}_X=d_{\rm dR}\bi_X+\bi_Xd_{\rm dR}=[d_{\rm dR},\bi_X]$;

\item $\bi_{[X,Y]}={\mathcal L}_X \bi_Y-\bi_Y{\mathcal L}_X=[{\mathcal
L}_X,\bi_Y]=[\bi_X,{\mathcal L}_Y]$;

\item $[\bi_X,\bi_Y]=0$.
\end{enumerate}
Note that the first Cartan formula above actually states that
$[d,\bi]={\mathcal L}$. Indeed  ${\mathcal X}(M)$ is
concentrated in  degree
zero and then its differential is trivial.
\end{example}

\begin{example}
Let $L$ be a DGLA and endowe  the complex $L[-1]$ with the trivial
bracket. Then the identity map
\[ \Id\in \Hom^{-1}(L[-1],L)\]
is a weak Cartan homotopy for the trivial map $0$.
In fact $d_{L[-1]}=-d_{L}$ and then $d_L(a)+d_{L[-1]}(a)=0$.
\end{example}

\begin{remark}\label{rem.cambiobasepercartan}
The composition of a (weak) Cartan homotopy with a morphism of DGLAs is a
(weak) Cartan homotopy.
If $\bi\colon L\to M[-1]$ is a Cartan
homotopy and
$\Omega$ is a differential graded-commutative algebra, then
its natural extension
\[\bi\otimes \operatorname{Id}\colon L\otimes \Omega\to (M\otimes
\Omega)[-1],\qquad
a\otimes \omega\mapsto \bi_a\otimes\omega,\]
is  a (weak) Cartan homotopy.
\end{remark}

\begin{remark}\label{rem.cartanhomotopyfor}
By definition, $\bl$ is the differential of $\bi$ in the complex
$\Hom^*(L,M)$ and so $\bi$ is a homotopy between $\bl$ and the
trivial map. Then the map $\bl\colon L\to M$ is a null-homotopic
morphism of DGLAs and
\[
\bi\colon L\to (\coker \bl) [-1],\qquad \bi\colon\ker\bl\to M[-1]
\]
are morphisms of differential graded vector spaces.
\end{remark}

\begin{proposition}\label{prop.cartan}  Let $\bl\colon L\to
M$ be a DGLA morphism, and let $\pi_1\colon
\widetilde{C}(\bl)\to L$ be the projection on the first
factor. A weak Cartan homotopy $\bi\colon L\to M[-1]$ for
$\bl$ is the datum of a linear $L_\infty$-morphisms $L\to
\widetilde{C}(\bl)$ lifting the identity of $L$. In
particular, if $\bi\colon L\to M[-1]$ is a weak Cartan
homotopy for $\bl$, then the map
$a\mapsto (a,e^{\bi_a})$ induces a natural transformation of
Maurer-Cartan functors
$\MC_L\to
\MC_{\bl}
$, and consequently a natural transformation of deformation
functors $\Def_L\to
\Def_{\bl}$.
\end{proposition}

\begin{proof}
A linear map $\tilde{\bi} \colon L\to \widetilde{C}({\bl})$
lifting the identity of $L$ has the form
$ \tilde{\bi}(a)=(a,\bi_a)$, with $\bi\colon L\to M[1]$.
By the explicit expression
for the higher brackets
\[\mu_n\colon\bigwedge^nC_{\chi}\to
C_{\chi}[2-n],\qquad n\ge 2,\]
 defining the $L_\infty$-algebra structure
$\widetilde{C}(\chi)$,  it is straightforward to check that
$\tilde{\bi}$ is a morphism of complexes commuting
 with every bracket if and only if $\bi$ is a weak Cartan homotopy.
Indeed, $\tilde{\bi}(dx)=\mu_1(\tilde{\bi}(x))$ is  the identity
$\bi_{da}=(-d)\bi_a+\bl_a$; the identity
$\tilde{\bi}([x,y])=\mu_2(\tilde{\bi}(x)\wedge
\tilde{\bi}(y))$
is
\[\bi_{[a,b]}=\frac{1}{2}[\bi_a,\bl_b]+
\frac{(-1)^{\deg(a)}}{2}[\bl_a,\bi_b]=[\bi_a,\bl_b],\]
and
$\mu_n(\tilde{\bi}(x_1)\wedge\cdots\wedge
\tilde{\bi}(x_n))=0$, for any $x_1,x_2,\dots,x_n$ and any
$n\ge 3$, if and only if
\[
\sum_{\sigma\in
S_3}\pm[\bi_{x_{\sigma(1)}},
[\bi_{x_{\sigma(2)}},\bl_{x_{\sigma(3)}}]=0
\]
for any $x_1,x_2,x_3$.
\par

Since the $L_\infty$-morphism $\tilde{\bi}$ is linear, the map
$l\mapsto (l,{\bi}_l)$ is a morphism of Maurer-Cartan functors
$\MC_L\to \MC_{\widetilde{C}(\bl)}$. To conclude the proof,
compose this morphism with the isomorphism
$\MC_{\widetilde{C}(\bl)}\xrightarrow{\sim} \MC_{\bl}$ given by
$(l,m)\mapsto (l,e^m)$.
 \end{proof}

\begin{corollary}\label{cor.cartan}
Let $\bi\colon N\to M[-1]$ be a Cartan homotopy for
$\bl\colon N\to M$, let $L$ be a subDGLA of $M$ such
that $\bl(N)\subseteq L$, and let $\chi\colon
L\hookrightarrow M$ be the inclusion. Then the linear map
\[
\Phi\colon N\to \widetilde{C}(\chi),\qquad
\Phi(a)=(\bl_a,\bi_a)
\]
is a linear  $L_{\infty}$-morphism. In particular, the
map $a\mapsto (\bl_a,e^{\bi_a})$ induces a natural
transformation of Maurer-Cartan functors $\MC_N\to \MC_{\chi}
$, and consequently a natural transformation of deformation
functors $\Def_N\to
\Def_{\chi}$.
\end{corollary}
\begin{proof}
We have a commutative diagram
of differential graded Lie algebras
\[
\xymatrix{
      N \ar[r]^{\bl}
\ar[d]_{\bl} &
L \ar[d]^{\chi}\\
    M  \ar@{=}[r] &
M\\
    }
\]
inducing an $L_\infty$-morhism $\widetilde{C}(\bl)\to
\widetilde{C}(\chi)$. Composing this morphism with the
$L_\infty$-morphism $\tilde{\bi}\colon
N\to\widetilde{C}(\bl)$ given by
Proposition~\ref{prop.cartan}, one gets the
$L_\infty$-morphism $\Phi$.
\end{proof}

\section{Polyvector fields and generalized periods}

The notion of Cartan homotopy generalizes immediately to sheaves
of DGLAs. In this section we give another example of Cartan
homotopy which we will use later.

Let $X$ be a complex manifold and  denote by:\begin{itemize}

\item $T_{X,\mathbb{C}}=T^{1,0}_X\oplus T^{0,1}_X$ the complexified
differential tangent
bundle.

\item $T_X\simeq T^{1,0}_X$ the holomorphic tangent bundle.

\item $\mathcal{A}_X^{p,q}$ the sheaf of differentiable $(p,q)$-forms
and by $\mathcal{A}_X^{p,q}(E)$ the sheaf of $(p,q)$-forms
with values in a holomorphic vector bundle $E$.

\item ${A}_X^{p,q}$  and ${A}_X^{p,q}(E)$ the vector spaces of
global sections
of $\mathcal{A}_X^{p,q}$ and $\mathcal{A}_X^{p,q}(E)$ respectively.

\end{itemize}

The direct sum
\[\mathcal{A}_X=\mathop{\oplus}_i\mathcal{A}^{i}_X,\quad\text{ where }\quad
\mathcal{A}^{i}_X=\mathop{\oplus}_{p+q=i}\mathcal{A}^{p,q}_X,\]
endowed with the wedge product $\wedge$, is a sheaf of graded
algebras; we denote by
${\mathcal H}om^{a,b}(\mathcal{A}_X,\mathcal{A}_X)$ the
sheaf of its
$\mathbb{C}$-linear endomorphisms of ${\mathcal A}_X$ of
bidegree
$(a,b)$. Notice that $\de$ and $\debar$ are global sections
of
${\mathcal H}om^{1,0}(\mathcal{A}_X,{\mathcal A}_X)$ and
${\mathcal H}om^{0,1}(\mathcal{A}_X,{\mathcal A}_X)$
respectively.  The direct sum
${\mathcal
H}om^{*}(\mathcal{A}_X,{\mathcal
A}_X)=\oplus_k\oplus_{a+b=k}{\mathcal
H}om^{a,b}(\mathcal{A}_X,{\mathcal A}_X)$ is a sheaf of
graded associative algebras, and so a sheaf of
differential graded Lie algebras with the
natural bracket
\[ [f,g]=fg-(-1)^{\deg(f)\deg(g)}gf\]
and  differential $[d,-]=[\de+\debar,-]$.\par

For any integer $(a,b)$ with $a\leq 0$ and $b\geq 0$,
let ${\mathcal G}erst^{a,b}_X$ be the sheaf
\[
{\mathcal G}erst^{a,b}_X={\mathcal
A}_X^{0,b}(\wedge^{-a}T_X).
\]
The direct sum ${\mathcal
G}erst^{*}_X=\oplus_k\oplus_{a+b=k}{\mathcal G}erst^{a,b}_X$
is a sheaf of differential Gerstenhaber algebras, with  the
wedge product
\[
\wedge\colon {\mathcal G}erst^{a_1,b_1}_X\otimes {\mathcal
G}erst^{a_2,b_2}_X\to{\mathcal G}erst^{a_1+a_2,b_1+b_2}_X
\]
as graded commutative product, 
the degree 1 differential
\[
\debar\colon {\mathcal G}erst^{a,b}_X\to{\mathcal
G}erst^{a,b+1}_X
\]
defined in local coordinates by the formula
\[ \debar\left(\phi\frac{\de~}{\de
z_I}\right)=\debar(\phi)\frac{\de~}{\de z_I},\qquad \phi\in
\mathcal{A}_X^{0,*},\] and the
degree 1 bracket
\[
[\;,\;]_{\rm G}\colon {\mathcal G}erst^{a_1,b_1}_X\otimes {\mathcal
G}erst^{a_2,b_2}_X\to{\mathcal G}erst^{a_1+a_2+1,b_1+b_2}_X
\]
defined in local coordinates by the formula
\[\left[fd\bar{z}_{I}\desude{~}{z_{H}},
gd\bar{z}_{J}\desude{~}{z_{K}}\right]_{\rm G}=
d\bar{z}_{I}\wedge
d\bar{z}_{J}\left[f\desude{~}{z_{H}},
g\desude{~}{z_{K}}\right]_{\rm SN}.\]
Here $[\,,\,]_{\rm SN}$ denotes the Schouten-Nijenhuis
 bracket on ${\mathcal A}_X^0(\wedge^*T_X)$, i.e., the odd
graded Lie bracket obtained extending the usual Lie bracket
on ${\mathcal A}_X^0(T_X)$ by imposing
\[
[\xi,f]_{\rm SN}=\xi(f), \qquad \xi\in {\mathcal
A}^0(T_X),\quad f\in {\mathcal A}_X^0;
\]
and the odd graded Poisson identity
\[
[\xi,\eta_1\wedge\eta_2]_{\rm SN}=
[\xi,\eta_1]_{\rm SN}\wedge \eta_2+(-1)^{\deg(\xi)(\deg(\eta_1)-1)}\eta_1
\wedge [\xi,\eta_2]_{\rm SN}.
\]

The contraction of differential forms with vector fields is used to
define an injective morphisms of sheaves of bigraded vector
spaces: the \emph{contraction map}
\[ \bi\colon {\mathcal G}erst^{a,b}_X\to
{\mathcal H}om^{a,b}(\mathcal{A}_X,{\mathcal
A}_X),\qquad
\xi\mapsto \bi_{\xi},\quad \bi_{\xi}(\omega)=\xi\contr \omega,\]%
The contraction map is actually a morphism of sheaves of
bigraded associative algebras:
\[
\bi_{\xi\wedge\eta}=\bi_\xi\bi_\eta
\]
In particular, since $({\mathcal G}erst^{*}_X,\wedge)$ is a
graded commutative algebra, we obtain
\[
[\bi_\xi,\bi_\eta]=0, \qquad \forall \xi,\eta\in {\mathcal
G}erst^{*}_X.
\]
Note that iterated contractions give a symmetric map
\begin{align*}
\bi^{(n)}\colon\bigodot^n{\mathcal G}erst^{*}_X&\to {\mathcal
H}om^{*}(\mathcal{A}_X,{\mathcal A}_X)\\
\xi_1\odot\xi_2\odot\cdots\odot\xi_n&
\mapsto\bi_{\xi_1}\bi_{\xi_2}\cdots\bi_{\xi_n}
\end{align*}
 Since ${\mathcal G}erst^{*}_X$
is a sheaf of differential graded Gerstenhaber algebras, its
desuspension
\[
{\mathcal P}oly^{*}_X={\mathcal G}erst[-1]_X^*
\]
is a sheaf of differential graded Lie algebras. Note that, due
to the shift, the differential $D$ in ${\mathcal
P}oly^{*}_X$ is $-\debar$, i.e., in local coordinates
\[
D\colon {\mathcal P}oly^{k}_X\to{\mathcal P}oly^{k+1}_X
\]
is given by the formula
\[ D\left(\phi\frac{\de~}{\de z_I}\right)=-\debar(\phi)\frac{\de~}{\de
z_I},\qquad \phi\in \mathcal{A}_X^{0,*},\]

The contraction map $\bi\colon {\mathcal G}erst^{*}_X\to {\mathcal
H}om^{*}(\mathcal{A}_X,{\mathcal A}_X)$ can be seen as a linear
map
\[
\bi\colon{\mathcal P}oly^{*}_X\to {\mathcal H}om^{*}
(\mathcal{A}_X,{\mathcal A}_X)[-1].
\]
More in general, via the decalage isomorphism, the iterated
contraction is a graded antisymmetric map
\[
\bi^{(n)}\colon\bigwedge^n{\mathcal P}oly^{*}_X\to
{\mathcal H}om^{*}(\mathcal{A}_X,{\mathcal A}_X)[-n].
\]

\begin{lemma}\label{lem.cartanequalities}
In the notation above,
for every $\xi,\eta\in {\mathcal P}oly^{*}_X$  we have
\[\bi_{D\xi}=-[\debar, \bi_\xi],\qquad
\bi_{[\xi,\eta]}=[\bi_\xi,[\de,\bi_\eta]],\qquad
[\bi_\xi,\bi_\eta]=0.\]
\end{lemma}

\begin{proof}
The third equation has been proved above, and the first
equation is completely straightforward: it just expresses the
Leibniz rule for $\debar$. To prove the second equation, let
\[
\Phi(\xi,\eta)=\bi_{[\xi,\eta]}-[\bi_\xi,[\de,\bi_\eta]].
\]
Then, using
$\bi_{\xi\wedge\eta}=\bi_{\xi}\bi_{\eta}$, the
(shifted) odd Poisson identity
$[\xi,\eta_1\wedge\eta_2]=[\xi,\eta_1]\wedge\eta_2
+(-1)^{(\deg(\xi)-1)\deg(\eta_1)}\eta_1\wedge[\xi,\eta_2]$
and the third equation $[\bi_\xi,\bi_\eta]=0$, one finds
\[
\Phi(\xi,\eta_1\wedge
\eta_2)=\Phi(\xi,\eta_1)\bi_{\eta_2}+(-1)^{(\deg(\xi)-1)\deg(\eta_1)}
\bi_{\eta_1}\Phi(\xi,\eta_2)
\]
and
\[
\Phi(\xi_1\wedge
\xi_2,\eta)=\bi_{\xi_1}\Phi(\xi_2,\eta)+(-1)^{(\deg(\xi_2)(
\deg(\eta)-1)}
\Phi(\xi_1,\eta)\bi_{\eta_2}.
\]
Therefore, to prove $\Phi(\xi,\eta)=0$ for any $\xi,\eta$
one just needs to prove $\Phi(\xi,\eta)=0$ for $\xi,\eta\in
{\mathcal A}^0_X\cup
\{d\overline{z}_i\}\cup\{\partial/\partial z_j\}$, where
$z_1,\ldots,z_n$ are local holomorphic coordinates.
This is straightforward and it is left to
the reader: see also Lemma 7 of \cite{CCK} and Lemma 7.21
of \cite{defomanifolds}.
\end{proof}

\begin{corollary}\label{cor.cartanequalities}
The contraction map $\bi\colon {\mathcal P}oly^{*}_X\to
{\mathcal H}om^{*}(\mathcal{A}_X,\mathcal{A}_X)[-1]$ is a
Cartan homotopy and the induced morphism $\bl$ of sheaves of
differential graded Lie algebras is the \emph{holomorphic
Lie derivative}
\[ \bl\colon {\mathcal P}oly^{*}_X\to
{\mathcal H}om^{*}(\mathcal{A}_X,\mathcal{A}_X),\qquad
\xi\mapsto \bl_{\xi}=[\de,\bi_\xi],\quad \bl_{\xi}(\omega)=\de(\xi\contr
\omega)+(-1)^{\deg(\xi)}\xi\contr \de\omega.\]
Moreover, $\bl$ is an injective morphism of
sheaves.
\end{corollary}

\begin{proof}
As in Section \ref{sec.cartan}, let
$\bl_\xi=[d,\bi_\xi]+\bi_{D\xi}$. Since $\bi_{D\xi}=-[\debar,
\bi_\xi]$, we find $\bl_\xi=[\de,\bi_\xi]$. The identity
$\bi_{[\xi,\eta]}=[\bi_\xi,[\de,\bi_\eta]]$ then reads
$\bi_{[\xi,\eta]}=[\bi_\xi,\bl_\eta]$ and this, together with
$[\bi_\xi,\bi_\eta]=0$, tells us that the contraction map $\bi$ is
a Cartan homotopy and, consequently, the holomorphic Lie
derivative $\bl_\xi=[\de,\bi_\xi]$ is the induced morphism of
sheaves of differential graded Lie algebras. Injectivity of $\bl$
is easily checked in local coordinates.
\end{proof}

The Corollary~\ref{cor.cartanequalities}, applied to global
sections shows that the contraction  map
\[\bi\colon \operatorname{Poly}_X
=\oplus_{i,j} A^{0,i}(\wedge^{-j} T_X) \to
\Hom^{*}(\mathcal{A}_X,\mathcal{A}_X)[-1]\] is a Cartan homotopy,
as well as its restriction to the Kodaira-Spencer DGLA
$KS_X=\oplus_{i} A^{0,i}(T_X)$.

\begin{remark} The composition of the inclusion
$KS_X\hookrightarrow \operatorname{Poly}_X$ with the iterated
contraction $\bi^{(n)}\colon \wedge^n\operatorname{Poly}_X\to
\Hom^*({\mathcal A}_X,{\mathcal A}_X)[-n]$ induces in cohomology a
graded antisymmetric map $ \bigwedge^n H^*(KS_X) \to
\Hom^{*}(H^*(X,\mathbb{C}),H^*(X,\mathbb{C}))[-n]$. In
particular, from the isomorphism of graded vector spaces
$H^1(KS^*_X)\simeq H^1(T_X)[-1]$ and the decalage isomorphism, the
iterated contraction gives a symmetric morphism
\[
\bi^{(n)}\colon \bigodot^n H^1(T_X) \to \Hom^{0}(H^*(X,\mathbb{C}),
H^*(X,\mathbb{C})).
\]
It is well known \cite{Greencime} and easy to prove that the image
of $\bi^{(n)}$ consists of self-adjoint operators with respect the
cup product on $H^*(X,\mathbb{C})$.\par

Under the identification $H^*(X,\mathbb{C})=\oplus_{p,q}H^q(X,\Omega^p_X)$, 
when $\dim X=n$ the morphism
$\bi^{(n)}$ reduces to the Yukawa coupling \[ \bi^{(n)}\colon
\bigodot^n H^1(T_X) \to \bigodot^2 H^n(X,\mathcal{O}_X).
\]
\end{remark}

Now, as in Section~\ref{sec.example-kaehler}, consider the DGLA
\[ L=\{f\in \Hom^*(A_X,A_X)\mid f(\ker\partial)\subseteq
\partial A_X\},\]
and let $\chi\colon L\hookrightarrow \Hom^*(A_X,A_X)$ be the
inclusion. Since, $\bl(\operatorname{Poly}_X)\subseteq L$, by
Corollary~\ref{cor.cartan}, we have a natural transformation of
deformation functors $\Def_{\operatorname{Poly}_X}\to \Def_{\chi}$
induced, at the Maurer-Cartan level, by the map $\xi\mapsto
(\bl_\xi,e^{\bi_\xi})$.\par

The functor $\Def_{\operatorname{Poly}_X}$ is called the functor
of generalized deformations of $X$, see \cite{BK}, and we will
denote it by $\widetilde{\Def}_{X}$. We have shown in
Section~\ref{sec.example-kaehler} that there exists a natural
isomorphism $\psi\colon \Def_{\chi}\to \Aut_{H^*(X,\mathbb{C})}$. 
By these considerations, we obtain:

\begin{theorem}\label{thm:pre-periods}
The linear map
\[ \operatorname{Poly}_X\to \widetilde{C}(\chi),\qquad
\xi\mapsto(\bl_{\xi},\bi_{\xi})\]%
is a linear $L_{\infty}$-morphism and induces a natural
transformation of functors
\[ \Phi\colon \widetilde{\Def}_{X}\to \Aut_{H^*(X,\mathbb{C})},\]
given at the level of Maurer-Cartan functors by the map $\xi\mapsto
\psi_{{\bi}_\xi}$.
\end{theorem}

\begin{proposition}\label{prop.differential} Via the natural
identifications $H^1(\operatorname{Poly}_X)=\oplus_{i\geq
0}H^i(\wedge^iT_X)$ and $H^*(X,\mathbb{C})=
\oplus_{p,q}H^q(X,\Omega^p_X)$ given by the Dolbeault's
theorem and the $\de\debar$-lemma, the differential of $\Phi$,
\[
d\Phi\colon H^1(\operatorname{Poly}_X)\to 
\Hom^0(H^*(X,\mathbb{C}),H^*(X,\mathbb{C}))
\]
is identified with the contraction
\[
\left(\oplus_{i\geq 0}H^i(\wedge^iT_X)\right)\otimes
\left(\oplus_{p,q}H^q(X,\Omega^p_X)\right)\to
\oplus_{i,p,q}H^{q+i}(X,\Omega^{p-i}_X)
\]
\end{proposition}

\begin{proof} By Lemma~\ref{lem.cartanequalities} we have a
commutative diagram of differential complexes
\[
\xymatrix{
      (\Hom^{*-1}(A_X,A_X),-{\rm ad}_\debar) \ar[r]
 &
(\Hom^{*-1}(\ker\de,A_X/\de A_X),-{\rm ad}_d)
\\
    (\operatorname{Poly}_X,D)\ar[u]^{\bi}\ar[ur]_{\bi}&
\\
    }
\]
where we have used the fact that on $\ker\de$ and on
$A_X/\de A_X$ the differentials ${\rm ad}_d$ and ${\rm
ad}_\debar$ coincide. Using the identification
$H^*_\debar(A_X)=H^*(X,\mathbb{C})$ coming from
Dolbeault's theorem and the $\de\debar$-lemma, and by
Remark~\ref{rem.differential-psi}, the above commutative
diagram induces the commutative diagram in cohomology
\[
\xymatrix{
      \Hom^0(H^*(X,\mathbb{C}),H^*(X,\mathbb{C}))
 &
\Hom^0(H^*(\ker\de),H^*(A_X/\de
A_X))\ar[l]_{d\psi\,\,}
\\
    H^1(\operatorname{Poly}_X)\ar[u]^{\bi}\ar[ur]_{\bi}&
\\
    }
\]
Since, by Theorem~\ref{thm:pre-periods}, the differential of
$\Phi$ is $d\Phi=d\psi\circ\bi$, this ends the proof.
\end{proof}

As a corollary of Theorem~\ref{thm:pre-periods}, the linear map
$\xi\mapsto(\bl_\xi,\bi_\xi)$ induces a morphism of obstruction
spaces $H^2(\operatorname{Poly})\to\Hom^1(H^*(X,\mathbb{C}),
H^*(X,\mathbb{C}))$ commuting with $\Phi$ and obstruction
maps. The same argument of Proposition~\ref{prop.differential}
shows that this morphism is naturally identified with the
contraction
\[
\left(\oplus_{i\geq 0}H^{i+1}(\wedge^iT_X)\right)\otimes
\left(\oplus_{p,q}H^q(X,\Omega^p_X)\right)\to
\oplus_{i,p,q}H^{q+i+1}(X,\Omega^{p-i}_X).
\]
Since the deformation functor $\Aut_{H^*(X,\mathbb{C})}$ is
smooth, we obtain the following version of the so-called Kodaira
principle (ambient cohomology annihilates obstruction):

\begin{proposition}\label{prop:kodaira-principle}
The obstructions to extended deformations of a compact
K\"{a}hler manifold $X$ are contained in the subspace
\[
\bigoplus_{i\geq 0}\bigcap_{p,q}\ker\left(
H^{i+1}(\wedge^iT_X)\xrightarrow{\bi}
\Hom\left(H^q(X,\Omega^p_X),H^{q+i+1}(X,\Omega^{p-i}_X
\right)\right)
\]
of $H^2(\operatorname{Poly}_X)$.
\end{proposition}

As an immediate corollary we recover the fact  that extended
deformations of compact Calabi-Yau manifolds are unobstructed
\cite{BK}. Indeed, if $X$ is an $n$-dimensional compact Calabi-Yau
manifold, then for any $i\geq 0$ the contraction pairing
\[
H^{i+1}(\wedge^iT_X)\otimes H^{0}(X,\Omega^n_X)\to
H^{i+1}(X,\Omega^{n-i}_X)
\]
is nondegenerate.

\bigskip
\section{Restriction to classical deformations}
\label{sec.kodairaspencer}

Let $(X,{\mathcal O}_X)$ be a complex manifold.
It is well known that the infinitesimal deformations of the
complex structure of $X$ are governed by the Kodaira-Spencer
DGLA of $X$. More precisely, there is a natural isomorphism
of deformation functors
\[
\Def_{KS_X}\xrightarrow{\sim}\Def_X
\]
mapping a Maurer-Cartan element $\xi\in
A^{0,1}_X(T_X)$
to the complex manifold $(X,{\mathcal O}_\xi)$,
 where the structure sheaf ${\mathcal
O}_\xi$ is defined by
\[
{\mathcal O}_\xi= \ker \{\debar_\xi\colon
{\mathcal A}^{0,0}_X\to {\mathcal
A}^{0,1}_X\}=\{f\in
\mathcal{A}^0_X\mid (\debar+\bl_{\xi})f=0\},
\]
see  \cite{clemens}, \cite{GoMil2},
\cite[Ex. 3.4.1]{K} or \cite{Iaconophd}. The above
equations have to be intended as identities among
functors of Artin rings; namely, they mean that for any local
Artin algebra $(B,{\mathfrak m}_B)$ the Kuranishi data
$\xi\in
\MC_{KS_X}(B)\subseteq A^{0,1}_X(T_X)\otimes\mathfrak{m}_B$
are mapped to the family
$(X,{\mathcal O}_\xi)$ of complex manifolds over $\Spec(B)$,
 whose structure sheaf ${\mathcal
O}_\xi$ is defined by
\[
{\mathcal O}_\xi= \ker \{\debar_\xi\colon
{\mathcal A}^{0,0}_X\otimes B\to {\mathcal
A}^{0,1}_X\otimes B\}=\{f\in
\mathcal{A}^0_X\otimes B\mid (\debar+\bl_{\xi})f=0\}.
\]

Let now
\[
A_X=F^0_\xi\supseteq F^1_\xi\supseteq\cdots\] be the Hodge
filtration of differential forms on the complex manifold
$(X,{\mathcal O}_\xi)$, i.e. for every $m\ge 0$ $F^m_{\xi}$ is the
complex of global sections of the differential ideal sheaf
$\mathcal{F}^m_{\xi} \subseteq \mathcal{A}_X$ generated by
$(d\mathcal{O}_{\xi})^m$. Again, here we are writing ${\mathcal
A}_X$ for the functor of Artin rings defined by $B\mapsto
{\mathcal A}_X\otimes B$. If $X$ is a compact K\"ahler manifold,
the  cohomology of $(F^m_{\xi},d)$ naturally embeds into the
cohomology of $(A_X,d)$. Since the dimension of $H^*(F^m_\xi,d)$
is independent of $\xi$, one can look at $H^*(F^m_\xi,d)$ as  a
different linear embedding of $H^*(F^m,d)$ into $H^*(X;{\mathbb
C})$, hence $\xi\mapsto H^*(F^m_\xi,d)$ is a map
\[
\Def_X\to \Grass_{H^*(F^m),H^*(X;{\mathbb C})},
\]
called the $m^{\rm th}$ \emph{period map}.\par

The inclusion of DGLAs $KS_X\hookrightarrow \operatorname{Poly}_X$
induces an embedding of deformation functors $\Def_X\to
\widetilde{\Def}_X$. Hence, the restriction of $\Phi$ to $\Def_X$
is a natural transformation
\[
\Phi\colon \Def_X\to  \Aut_{H^*(X,\mathbb{C})}.
\]

\begin{theorem} 
For any $m\geq 0$, the map $\Phi\colon\Def_X\to \Aut_{H^*(X,\mathbb{C})}$ 
lifts the $m^{\rm th}$ 
period map $\Def_X\to \Grass_{H^*(F^m),H^*(X;{\mathbb C})}$.
\end{theorem}

\begin{proof}
Let $\xi$ be a Maurer-Cartan element in $KS_X$. Then $(\bl_\xi,e^{\bi_\xi})\in \MC_\chi$ is a Maurer-Cartan element with $\bl_\xi$ of bidegree $(0,1)$. Let $[\omega]$ be an element in $H^*(F^m)$. To compute $\psi_{\xi}[\omega]$ we pick a $\de$-closed representative, for the class $[\omega]$, which we can assume to be $\omega$, and then we take the cohomology class of a $d$-closed representative of $e^{\bi_\xi}\omega$ in ${e^{\bi_\xi}(\ker\de)}/{e^{\bi_\xi}(\de
A_X)}$, i.e., we have $\psi_\xi[\omega]=[ e^{\bi_\xi}
(\omega-\de \beta)]$ for any $\beta\in A_X$ such that $de^{\bi_\xi}
(\omega-\de \beta)=0$. For such a $\beta$ we have
\[
0=e^{-\bi_\xi}de^{\bi_\xi}
(\omega-\de \beta)=-\debar\de\beta+\bl_\xi(\omega)-\bl_\xi(\de\beta).
\]
and so
\[
(\debar+\bl_\xi)\de\beta=\bl_\xi(\omega).
\]
Write $\eta^{}_{<m}$ and $\eta^{}_{\geq m}$ for the components of a differential form $\eta$ in 
$\oplus_{i<m}A_X^{i,*}$ and in $\oplus_{i\geq m}A_X^{i,*}$, respectively.  
Since both $\bl_\xi$ and $(\debar+\bl_\xi)$ are homogeneous of bidegree $(0,1)$, we have
\[
\bl_\xi(\omega)=(\bl_\xi(\omega))_{\geq m}^{}=((\debar+\bl_\xi)\de\beta)_{\geq m}^{}
=(\debar+\bl_\xi)\de(\beta_{\geq m-1}^{}).
\]
Hence $\psi_{\bi_\xi}[\omega]=[ e^{\bi_\xi}
(\omega-\de (\beta_{\geq m-1}^{}))]\in H^*(e^{\bi_\xi}F^m)$, and so
\[
\Phi_\xi(H^*(F^m))=H^*(e^{\bi_\xi}F^m).
\]
On the other hand, the period of the infinitesimal deformation
$\mathcal{O}_{\xi}=\ker(\debar+\bl_{\xi})$ is
$H^*(F^m_{\xi})\subseteq H^*(A_X)$, where $F^m_{\xi}$ is the
complex of global sections
of the differential ideal sheaf $\mathcal{F}^m_{\xi}\subseteq
\mathcal{A}_X$ generated
by
$(d\mathcal{O}_{\xi})^m$.
Since $e^{\bi_{\xi}}$ is the identity on $\mathcal{A}_X^{0,0}$, by
Lemma~\ref{lem.newdifferential} we can write
\[
e^{-\bi_{\xi}}(d\mathcal{O}_{\xi})=e^{-\bi_{\xi}}de^{\bi_{\xi}}\mathcal{O}_{\xi}=
(\de+\debar+\bl_{\xi})\mathcal{O}_{\xi}=
\de \mathcal{O}_{\xi}\subseteq \de\mathcal{A}_X^{0,0}
\subseteq \mathcal{A}_X^{1,0}.\]
Since $e^{\bi_{\xi}}\colon \mathcal{A}_X\to
\mathcal{A}_X$ is a
morphism of sheaves of differential graded commutative algebras, we get $e^{-\bi_{\xi}}(\mathcal{F}^m_{\xi})\subseteq \mathcal{F}^m$ and then, by rank considerations,
$e^{\bi_{\xi}}(\mathcal{F}^m)=\mathcal{F}^m_{\xi}$.
Hence
\[
\Phi_\xi(H^*(F^m))=H^*(\mathcal{F}^m_{\xi})
\]
\end{proof}

\begin{remark}
It has been shown in \cite{Periods} that also the $m^{\rm th}$
period map
$\Def_X\to
\Grass_{H^*(F^m),H^*(X;{\mathbb C})}$ is induced by an
$L_\infty$-morphism. Namely, let $\chi_m\colon
L_m\hookrightarrow \Hom^*(A_X,A_X)$ the inclusion of the
subalgebra
\[
L_m=\{f\in \Hom^*(A_X,A_X)\,|\,f(F^m)\subseteq F^m\}
\]
in the DGLA of endomorphisms of $A_X$. Then
$\Def_{\chi_m}\simeq \Grass_{H^*(F^m),H^*(X)}$ and the map
$\xi\mapsto (\bl_\xi,\bi_\xi)$ is an $L_\infty$-morphism
between $KS_X$ and $C_{\chi_m}$ inducing the $m^{\rm th}$
period map.
\end{remark}

\end{document}